\documentclass[reqno,A4paper]{amsart}
\usepackage{amsmath, amsthm}\usepackage{enumerate}\usepackage{amssymb}\usepackage{bbold}\usepackage{color}
\usepackage{bbm}\usepackage{dsfont}\usepackage{color}\usepackage{xcolor}
\usepackage{mathrsfs} \usepackage{xcolor}
\usepackage{eqlist}
\usepackage{array}
\theoremstyle{plain}
\newtheorem{theorem}{Theorem}[section]

\newtheorem{lemma}[theorem]{Lemma}
\newtheorem{corollary}[theorem]{Corollary}

\newtheorem{definition}[theorem]{Definition}

\newtheorem{example}[theorem]{\textit{Example}} 

\numberwithin{equation}{section}

\newcommand{\oc}{\xrightarrow[]{o}}
\newcommand{\dopc}{\downarrow^{st_{op}}}
\newcommand{\doc}{\downarrow^{st_o}}

\newcommand{\socp}{\xrightarrow{st_{op}}}
\newcommand{\soc}{\xrightarrow{st_o}}

\begin{document}
\title[Statistical order convergence of operators]
{\small Statistical order convergence of operators on Riesz Spaces}
\author[A. Ayd\i n, E. Bayram, İ. Ayd\i n]{Abdullah Ayd\i n$^{1,*}$, Erdal Bayram$^2$, İshak Ayd\i n$^3$}

\newcommand{\acr}{\newline\indent}
\address{\llap{1\,} Department of Mathematics\acr
Mu\c{s} Alparslan University\acr
Mu\c{s}, Türkiye}
\email{a.aydin@alparslan.edu.tr}

\address{\llap{2\,} Department of Mathematics\acr
Tekirdağ Namık Kemal University\acr
Tekirdağ, Türkiye}
\email{ebayram@nku.edu.tr }

\address{\llap{3\,} Department of Mathematics\acr
Mu\c{s} Alparslan University\acr
Mu\c{s}, Türkiye}
\email{ishakaydin002@gmail.com}

\subjclass[2020]{46A40, 47B60, 40A35, 40A05}
\keywords{Statistical order convergence, order bounded operator, operator sequence, Riesz space, order convergence. \acr $^*$Corresponding Author}
\begin{abstract}
{This paper introduces statistical order convergence and its pointwise variant for sequences of order bounded operators between Riesz spaces. We establish fundamental properties: uniqueness of the limit, stability under lattice operations, and a characterization via natural density linking it to classical order convergence. Explicit examples show that statistical order convergence is strictly weaker than order convergence, confirming that this concept provides a proper extension of operator-theoretic convergence notions. The results preserve essential lattice structures and open avenues for further research in unbounded convergence and Banach lattice theory.}
\end{abstract}
\maketitle
\section{Introduction and Preliminaries}\label{Sec:1}
The concept of statistical convergence was introduced independently by Steinhaus \cite{St} and Fast \cite{Fast} in 1951, and has since been widely investigated in summability theory (see, e.g., \cite{Et,Con,Et1,Fridy,Trip}). In the context of function spaces, early studies were conducted by Moore \cite{Moor}, who utilized a scale function, and Chittenden \cite{Chit}, who investigated statistical convergence of continuous functions with respect to relatively uniform convergence. For further details on statistical convergence of functions, we refer the reader to \cite{BDK,DO,EDS,KDD,Mor}. Parallel to these developments, the theory of Riesz spaces (or vector lattices), originating from the foundational study of Riesz \cite{Riez}, has garnered significant attention due to its broad applications in analysis, economics, operator theory, and measure theory (see, e.g., \cite{AB1,AB2,AEG,BW,LZ,Vu,Za}). Although order convergence and related convergence notions in Riesz spaces are generally non-topological \cite{Gor}, a rich operator theory can be established within this framework without relying on topological structures. Consequently, the study of statistical convergence within Riesz spaces has attracted considerable attention in recent years, combining ideas from summability theory and ordered vector spaces.

Statistical convergence for sequences in Riesz spaces was first introduced by Şençimen and Pehlivan \cite{SP} using order convergence and subsequently extended by Aydın \cite{Aydn1,Aydn2,AE}. However, while statistical convergence of sequences has been extensively studied, the corresponding theory for operators between Riesz spaces remains largely unexplored. This paper aims to fill this gap by introducing and systematically investigating statistical order convergence for sequences of order bounded operators.

We now recall some essential definitions and results regarding Riesz spaces. A Riesz space is an ordered vector space with the lattice property; that is, for any two elements $u$ and $v$, the infimum $u \wedge v$ and the supremum $u \vee v$ exist. A Riesz space is called Dedekind complete if every non-empty subset that is bounded from above has a supremum. The Archimedean property asserts that $\frac{1}{n}u \downarrow \theta$ for all positive elements $u$, where $\theta$ denotes the zero element. Throughout this paper, unless otherwise specified, all Riesz spaces are assumed to be real and Archimedean. A sequence $(u_n)$ is called increasing (denoted by $u_n \uparrow$) if $u_1 \leq u_2 \leq \cdots$, and decreasing (denoted by $u_n \downarrow$) if $u_1 \geq u_2 \geq \cdots$. The notation $u_n \downarrow u$ signifies that $u_n \downarrow$ and $\inf u_n = u$, while $u_n \uparrow u$ indicates that $u_n \uparrow$ and $\sup u_n = u$.

\begin{definition} \cite[Thm. 16.2]{LZ}
Let $\mathcal{L}$ be a Riesz space and $(u_n)$ be a sequence in $\mathcal{L}$. The sequence $(u_n)$ is said to be \textit{order convergent} to $u$, denoted by $u_n \xrightarrow{o} u$, if there exists a sequence $v_n \downarrow \theta$ such that $|u_n - u| \leq v_n$ for all $n \in \mathbb{N}$.
\end{definition}

In a Riesz space $\mathcal{L}$, the set $\{u \in \mathcal{L} : -a \leq u \leq a\}$, where $a \in \mathcal{L}_+$, is called an \emph{order interval} and is denoted by $[-a, a]$. A subset $A \subseteq \mathcal{L}$ is said to be \emph{order bounded} if there exists $a \in \mathcal{L}_+$ such that $A \subseteq [-a, a]$. Throughout this paper, the term \emph{operator} will denote a linear map between two Riesz spaces.

\begin{definition}
Let $S$ be an operator between Riesz spaces $\mathcal{L}$ and $\mathcal{M}$.
\begin{enumerate}[(a)]
\item If $S$ maps order bounded subsets of $\mathcal{L}$ into order bounded subsets of $\mathcal{M}$, then $S$ is called an \textit{order bounded operator}.
\item If $S(u_n) \oc S(u)$ in $\mathcal{M}$ for any sequence $(u_n)$ in $\mathcal{L}$ with $u_n \oc u$, then $S$ is called a \textit{$\sigma$-order continuous operator}.
\end{enumerate}
\end{definition}

In the theory of Riesz spaces, the prefix $\sigma$ is typically associated with sequences and countable quantities. Throughout this paper, $\mathcal{L}$ and $\mathcal{M}$ denote Riesz spaces. The set of all order bounded operators from $\mathcal{L}$ to $\mathcal{M}$ is denoted by $\mathcal{L}_b(\mathcal{L}, \mathcal{M})$. This space becomes an ordered vector space when equipped with the partial order defined by $S \leq T$ if and only if $S(u) \leq T(u)$ for all $u \in \mathcal{L}_+$. Moreover, if $\mathcal{M}$ is Dedekind complete, then $\mathcal{L}_b(\mathcal{L}, \mathcal{M})$ is itself a Dedekind complete Riesz space (see, e.g., \cite[Thm. 1.67]{AB1}). 

Next, we recall some fundamental notions of statistical convergence. The \emph{natural density} of a subset $J$ of positive integers is defined by$$\delta(J):=\lim_{n \to \infty} \frac{1}{n} |\{j\in J:j\leq n\}|,$$provided the limit exists (see \cite{Fridy,Trip}). Accordingly, a real sequence $(r_n)$ is said to be \textit{statistically convergent} to $r$ if, for every $\varepsilon > 0$,
$$
\lim_{n\to \infty} \frac{1}{n} |\{j\leq n:|r_j- r| \geq \varepsilon\}| = 0.
$$
To motivate our main results regarding operators, we close this section by recalling the classical concepts of statistical convergence in the space of continuous functions. Let $C(X)$ denote the space of all real-valued continuous functions on a compact subset $X$ of real numbers, equipped with the supremum norm $\|f\|_\infty = \sup_{x \in X} |f(x)|$. We adopt the following definitions from \cite{BDK,Chit,DO,EDS,KDD}.
\begin{definition}
Let $(f_n)$ be a sequence in $C(X)$. For any $x \in X$ and $\varepsilon > 0$, let us define the sets
$$
\Psi_n(x, \varepsilon) := |\{k \leq n : |f_k(x) - f(x)| \geq \varepsilon\}|
$$
and
$$
\Phi_n(\varepsilon) := |\{k \leq n : \|f_k - f\|_\infty \geq \varepsilon\}|.
$$
\begin{enumerate}[(i)]
\item If $\lim_{n \to \infty} \frac{1}{n}\Psi_n(x, \varepsilon) = 0$ for every $x \in X$ and $\varepsilon > 0$, then $(f_n)$ is said to be {\em statistically pointwise convergent} to $f$ on $X$.
\item If the convergence $\lim_{n \to \infty} \frac{1}{n}\Psi_n(x, \varepsilon) = 0$ is uniform with respect to $x \in X$ for each $\varepsilon > 0$, then $(f_n)$ is said to be {\em equi-statistically convergent} to $f$.
\item If $\lim_{n \to \infty} \frac{1}{n}\Phi_n(\varepsilon) = 0$ for every $\varepsilon > 0$, then $(f_n)$ is said to be {\em statistically uniformly convergent} to $f$.\end{enumerate}
\end{definition}

\section{Statistical monotonicity}\label{Sec:2}
The notion of statistical monotonicity was first introduced by Tripathy \cite{Trip} for real-valued sequences and was subsequently extended to sequences in Riesz spaces by Şençimen and Pehlivan \cite{SP}. Statistical monotonicity plays a fundamental role in the theory of statistical convergence. In the operator setting, however, monotonicity may be interpreted in two distinct ways, namely, pointwise monotonicity and monotonicity with respect to the operator order. The purpose of this section is to clarify the relationship between these two notions and to investigate their basic properties.

\begin{definition}
Let $(S_n)$ be a sequence in $\mathcal{L}_b(\mathcal{L},\mathcal{M})$. The sequence $(S_n)$ is said to be:
\begin{enumerate}[(i)]
\item \textit{statistically order pointwise decreasing} to an operator $S \in \mathcal{L}_b(\mathcal{L},\mathcal{M})$, denoted by $S_n \dopc S$, if, for each $u \in \mathcal{L}_+$, there exists a subset $J_u \subseteq \mathbb{N}$ with $\delta(J_u) = 1$ such that $(S_n(u))$ decreases to $S(u)$ along indices in $J_u$.
\item \textit{statistically order decreasing} to $S \in \mathcal{L}_b(\mathcal{L},\mathcal{M})$, denoted by $S_n \doc S$, if there exists a subset $J \subseteq \mathbb{N}$ with $\delta(J) = 1$ such that $(S_n)$ decreases to $S$ in $\mathcal{L}_b(\mathcal{L},\mathcal{M})$ along indices in $J$.
\end{enumerate}
\end{definition}

\begin{theorem}\label{implication}
Every statistically order decreasing sequence is statistically order pointwise decreasing to the same limit.
\end{theorem}

\begin{proof}
Suppose $S_n \doc S$. Then there exists a subset $J \subseteq \mathbb{N}$ with $\delta(J) = 1$ such that $(S_n)$ decreases to $S$ in the operator order along $J$. Let $u \in \mathcal{L}_+$ be fixed. Since $S_j \downarrow S$ on $J$, the sequence $(S_j(u))_{j \in J}$ is decreasing in $\mathcal{M}$. Furthermore, recalling that the infimum of a decreasing sequence of operators is attained pointwise (see, e.g., \cite[Thm. VIII.2.3]{Vu}), we have
$$
\inf_{j \in J} S_j(u) = (\inf_{j \in J} S_j)(u) = S(u).
$$
Thus, for every $u \in \mathcal{L}_+$, the sequence $(S_n(u))$ decreases to $S(u)$ along the set $J$ having density one. We conclude that $S_n \dopc S$.
\end{proof}

The converse of Theorem \ref{implication} fails in general. This stems from the fact that the intersection of an arbitrary family of sets with natural density one does not necessarily have density one.

\begin{theorem}\label{thm:operator-uniqueness}
Let $(S_n)$ be a sequence in $\mathcal{L}_b(\mathcal{L}, \mathcal{M})$. Then the following hold:
\begin{enumerate}[(i)]
\item If $S_n \dopc S$ and $S_n \dopc P$, then $S = P$.
\item If $S_n \doc S$ and $S_n \doc P$, then $S = P$.
\end{enumerate}
\end{theorem}

\begin{proof}
$(i)$ Assume $S_n \dopc S$ and $S_n \dopc P$. Let $u \in \mathcal{L}_+$ be arbitrary. By definition, there exist subsets $J_u, K_u \subseteq \mathbb{N}$ with $\delta(J_u) = \delta(K_u) = 1$ such that
$$
S_j(u)\downarrow S(u) \quad \text{and} \quad S_k(u)\downarrow P(u).
$$
on $J_u$ and $K_u$, respectively. Define $M_u := J_u \cap K_u$. Since the intersection of two sets of density one has density one, we have $\delta(M_u) = 1$. The sequence $(S_n(u))$ restricted to indices in $M_u$ is a decreasing subsequence of both $(S_j(u))_{j \in J_u}$ and $(S_k(u))_{k \in K_u}$. By the uniqueness of order limits in Riesz spaces, we obtain
$$
S(u) = \inf_{m \in M_u} S_m(u) = P(u).
$$
Since $S(u) = P(u)$ holds for every $u \in \mathcal{L}_+$, and operators are determined by their values on the positive cone, we conclude that $S = P$.

$(ii)$ Assume $S_n \doc S$ and $S_n \doc P$. By Theorem \ref{implication}, we have $S_n \dopc S$ and $S_n \dopc P$. Then part $(i)$ implies $S = P$.
\end{proof}

\begin{theorem}\label{thm:decomposition}
A sequence $(S_n)$ in $\mathcal{L}_b(\mathcal{L}, \mathcal{M})$ is statistically order decreasing to an operator $S$ if and only if there exists a decreasing sequence $(P_n)$ in $\mathcal{L}_b(\mathcal{L}, \mathcal{M})$ such that $P_n \downarrow S$ and $S_n = P_n$ for almost all $n$.
\end{theorem}

\begin{proof}
The necessity follows standard arguments and is omitted. For the sufficiency, assume there exists a sequence $(P_n)$ in $\mathcal{L}_b(\mathcal{L}, \mathcal{M})$ such that $P_n \downarrow S$ and a subset $J \subseteq \mathbb{N}$ with $\delta(J) = 1$ satisfying $S_j = P_j$ for all $j \in J$. Consider any indices $j, m \in J$ with $j < m$. Since $(P_n)$ is decreasing, we have
$$
P_m = S_m\le S_j = P_j.
$$
Thus, $(S_{n})$ is decreasing along $J$. Moreover, since $P_{n} \downarrow S$, we obtain
$$
\inf_{j \in J} S_j = \inf_{j \in J} P_j = S.
$$
Hence we get $S_n\doc S$ as required. 
\end{proof}

An analogous characterization holds for statistically order pointwise decreasing sequences. That is, $S_n \dopc S$ if and only if there is a sequence $(P_n)$ such that $P_n(u) \downarrow S(u)$ for all $u \in \mathcal{L}_+$ and $S_n = P_n$ for almost all $n$. The proof of the following lemma follows from straightforward order inequalities and is therefore omitted.
\begin{lemma}\label{lem:additivity}
The concepts of statistically order decreasing and statistically order pointwise decreasing convergence satisfy additivity and positive homogeneity. That is, if $S_n \doc S$, $P_n \doc P$, and $\alpha \ge 0$, then $S_n + P_n \doc S + P$ and $\alpha S_n \doc \alpha S$. Analogous results hold for statistically order pointwise decreasing sequences.
\end{lemma}

\begin{theorem}\label{thm:lattice-operations}
Assume that $\mathcal{M}$ is a Dedekind complete Riesz space and let $(S_{n})$ and $(P_{n})$ be sequences in $\mathcal{L}_{b}(\mathcal{L}, \mathcal{M})$. If $S_n \doc S$ and $P_n \doc P$, then the following properties hold:
\begin{enumerate}[(i)]
\item $S_{n} \vee P_n \doc S \vee P$.
\item $S_{n} \wedge P_n \doc S \wedge P$.
\item $S_{n}^{+} \doc S^{+}$.
\end{enumerate}
\end{theorem}

\begin{proof}
Since $\mathcal{M}$ is Dedekind complete, the space $\mathcal{L}_{b}(\mathcal{L}, \mathcal{M})$ is also a Dedekind complete Riesz space. This ensures that the infima of bounded monotonic sequences exist.

$(i)$ By hypothesis, there exist subsets $J, K \subseteq \mathbb{N}$ with natural density one such that $(S_n)$ decreases to $S$ along $J$, and $(P_n)$ decreases to $P$ along $K$. Let $M := J \cap K$. Since the intersection of two sets of density one has density one, we have $\delta(M)=1$. 
Since both $(S_m)_{m \in M}$ and $(P_m)_{m \in M}$ are decreasing sequences, it follows from standard Riesz space properties that their supremum $(S_m \vee P_m)_{m \in M}$ is also a decreasing sequence. Furthermore, in a Dedekind complete Riesz space, the infimum operation distributes over the supremum for decreasing sequences (see, e.g., \cite[Thm. 16.1(iv)]{LZ}). Therefore, we have
$$
\inf_{m \in M} (S_{m} \vee P_{m}) = (\inf_{m \in M} S_m) \vee (\inf_{m \in M} P_m) = S \vee P.
$$
This confirms that $S_{n} \vee P_n \doc S \vee P$.

$(ii)$ The proof for $S_{n} \wedge P_{n}$ is analogous. Since $(S_m \wedge P_m)$ is decreasing on $M$ and $\inf(S_m \wedge P_m) = (\inf S_m) \wedge (\inf P_m)$, the result follows.

$(iii)$ Recall that $S_{n}^{+} = S_{n} \vee \theta$. Since the constant zero sequence $(\theta)$ is trivially statistically order decreasing to $\theta$, the result follows immediately as a special case of part $(i)$.
\end{proof}

It is clear that the conclusions of Theorem \ref{thm:lattice-operations} also hold for statistically order pointwise decreasing sequences.

\begin{theorem}\label{thm:squeeze}
Let $(S_n)$, $(P_n)$, and $(Q_n)$ be sequences in $\mathcal{L}_b(\mathcal{L}, \mathcal{M})$, and let $T \in \mathcal{L}_b(\mathcal{L}, \mathcal{M})$. Assume that $(P_n)$ is decreasing and for every $n \in \mathbb{N}$,
$$
S_n \leq P_n \leq Q_n.
$$
If $S_n \dopc T$ and $Q_n \dopc T$, then $P_n \dopc T$.
\end{theorem}

\begin{proof}
Let $u \in \mathcal{L}_+$ be arbitrary. Since $S_n \dopc T$ and $Q_n \dopc T$, there exist subsets $J_u, M_u \subseteq \mathbb{N}$ with natural densities one such that $S_j(u)\downarrow T(u)$ on $J_u$ and $Q_m(u)\downarrow T(u)$ on $M_u$. Let $K_u := J_u \cap M_u$. Since the intersection of two sets of natural density one has density one, we have $\delta(K_u) = 1$. For every $k \in K_u$, since $S_k(u)$ decreases to $T(u)$, we have $S_k(u) \ge T(u)$. Combining this with the hypothesis, we obtain the inequalities:
$$
T(u) \le S_k(u) \leq P_k(u) \leq Q_k(u).
$$
Since $(Q_n)$ decreases to $T$ on $K_u$, passing to the infimum over $k \in K_u$ yields:
$$
T(u) \le \inf_{k \in K_u} P_k(u) \leq \inf_{k \in K_u} Q_k(u) = T(u).
$$
Thus, $\inf_{k \in K_u} P_k(u) = T(u)$. Since $(P_n)$ is decreasing by hypothesis, the subsequence $(P_k(u))_{k \in K_u}$ is decreasing. So we have $P_k(u) \downarrow T(u)$ on the set $K_u$ with $\delta(K_u)=1$. This implies $P_n \dopc T$.
\end{proof}

\begin{theorem}\label{thm:composition}
Let $(S_n)$ be a sequence in $\mathcal{L}_b(\mathcal{L}, \mathcal{M})$, and let $T: \mathcal{M} \to \mathcal{N}$ be a positive $\sigma$-order continuous operator. Then the following statements hold:
\begin{enumerate}[(i)]
\item If $S_n \dopc S$, then $T \circ S_n \dopc T \circ S$.
\item If $S_n \doc S$, then $T \circ S_n \doc T \circ S$.
\end{enumerate}
\end{theorem}

\begin{proof}
We prove $(ii)$ in detail; the proof of $(i)$ follows by an analogous pointwise argument. Assume $S_n \doc S$. Then there exists a subset $J \subseteq \mathbb{N}$ with $\delta(J) = 1$ such that $(S_j)$ decreases to $S$. 

First, let us establish monotonicity. Since $T$ is a positive operator, it is order preserving. Therefore, for any $j, m \in J$ with $j < m$, the inequality $S_m \leq S_j$ implies $T(S_m) \leq T(S_j)$, which means $T \circ S_m \leq T \circ S_j$. Thus, the sequence of operators $(T \circ S_j)_{j \in J}$ is decreasing.

Next, we identify the order limit. Fix an arbitrary $u \in \mathcal{L}_+$. Since $S_j \downarrow S$ on $J$, we have $S_j(u) \downarrow S(u)$ in $\mathcal{M}$. By the $\sigma$-order continuity of $T$, we obtain:
$$
\inf_{j \in J} (T \circ S_j)(u) = \inf_{j \in J} T(S_j(u)) = T(\inf_{j \in J} S_j(u)) = T(S(u)) = (T \circ S)(u).
$$
Since this equality holds for every $u \in \mathcal{L}_+$, and $(T \circ S_j)$ is decreasing, it follows that
$$
\inf_{j \in J} (T \circ S_j) = T \circ S.
$$
This implies that $T \circ S_n \doc T \circ S$.
\end{proof}

\begin{theorem}\label{thm:right-composition}
Let $(S_n)$ be a sequence in $\mathcal{L}_b(\mathcal{L},\mathcal{M})$ such that $S_n \doc S$. If $P: \mathcal{L} \to \mathcal{L}$ is a positive operator, then the sequence $(S_n \circ P)$ is statistically order decreasing to $S \circ P$.
\end{theorem}

\begin{proof}
Since $S_n \doc S$, there exists a subset $J \subseteq \mathbb{N}$ with natural density $\delta(J) = 1$ such that $S_j \downarrow S$ on $J$. Let $P$ be a positive operator on $\mathcal{L}$ and fix an arbitrary $u \in \mathcal{L}_+$. Since $P$ is positive, we have $v := P(u) \in \mathcal{L}_+$. Consequently, the sequence $(S_j(v))_{j \in J}$ is decreasing in $\mathcal{M}$. Furthermore, recalling that the infimum of a decreasing sequence of operators is determined pointwise (see \cite[Thm. VIII.2.3]{Vu}), we obtain
$$
\inf_{j \in J} (S_j \circ P)(u) = \inf_{j \in J} S_j(P(u)) = (\inf_{j \in J} S_j)(P(u)) = S(P(u)) = (S \circ P)(u).
$$
Therefore, for every $u \in \mathcal{L}_+$, we have $\inf_{j \in J} (S_j \circ P)(u) = (S \circ P)(u)$. Since the infimum of a decreasing sequence of operators is determined pointwise, we conclude that $\inf_{j \in J} (S_j \circ P) = S \circ P$, which implies $S_n \circ P \doc S \circ P$.
\end{proof}

\begin{theorem}\label{thm:wedge-zero}
Let $\mathcal{M}$ be a Dedekind complete Riesz space and let $T \in \mathcal{L}_b(\mathcal{L}, \mathcal{M})$ be a positive operator. If $S_n \doc \theta$, then $S_n \wedge T \doc \theta$.
\end{theorem}

\begin{proof}
Assume $S_n \doc \theta$. Then there exists a subset $J \subseteq \mathbb{N}$ with $\delta(J) = 1$ such that $(S_n)$ decreases to $\theta$ along $J$. Since $(S_n)$ decreases to zero on $J$, we have $S_j \ge \theta$ for all $j \in J$. Define $R_n := S_n \wedge T$. Since $T \ge \theta$ and $S_j \ge \theta$, it follows that $R_j \ge \theta$ for every $j \in J$.

Consider any indices $j, m \in J$ with $j < m$. Since the lattice operation $\wedge$ is order preserving and $S_m \le S_j$, we obtain:
$$
R_m = S_m \wedge T \le S_j \wedge T = R_j.
$$
Thus, the sequence $(R_n)$ is decreasing along $J$. Furthermore, for each $j \in J$, the inequalities $\theta \le S_j \wedge T \le S_j$ hold. Since $\inf_{j \in J} S_j = \theta$, passing to the infimum over $J$ yields:
$$
\theta \le \inf_{j \in J} (S_j \wedge T) \le \inf_{j \in J} S_j = \theta.
$$
Consequently, $\inf_{j \in J} R_j = \theta$. This concludes that $S_n \wedge T \doc \theta$.
\end{proof}
\section{Statistical order convergence}\label{Sec:3}
In this section, we define statistical order convergence for sequences in $\mathcal{L}_b(\mathcal{L}, \mathcal{M})$ and investigate their fundamental properties, such as uniqueness of the limit, the relationship with order convergence, and the behavior under lattice operations. We also provide an example to illustrate that statistical order convergence is strictly weaker than order convergence. Throughout this section, unless stated otherwise, we assume that $\mathcal{M}$ is a Dedekind complete Riesz space, which ensures that $\mathcal{L}_b(\mathcal{L}, \mathcal{M})$ is also Dedekind complete Riesz space (see, e.g., \cite[Thm. 1.18]{AB1}).
\begin{definition}
Let $(R_n)$ be a sequence in $\mathcal{L}_b(\mathcal{L},\mathcal{M})$. Then $(R_n)$ is called:
\begin{enumerate}[(i)]
\item \textit{statistically order convergent} to $R$, denoted by $R_n \soc R$, if there exists a sequence $(S_n)$ in $\mathcal{L}_b(\mathcal{L},\mathcal{M})$ such that $S_n \doc \theta$ and a subset $J \subseteq \mathbb{N}$ with $\delta(J)=1$ satisfying
$$
|R_j-R|\leq S_j
$$
for all $j\in J$.
\item \textit{statistically order pointwise convergent} to $R$, denoted by $R_n \socp R$, if, for any $u\in \mathcal{L}_+$, there exists a sequence $(S_n)$ in $\mathcal{L}_b(\mathcal{L},\mathcal{M})$ such that $S_n \dopc \theta$ and a subset $J_u \subseteq \mathbb{N}$ with $\delta(J_u)=1$ satisfying
$$
|R_j(u)-R(u)|\leq S_j(u)
$$
for all $j\in J_u$.
\end{enumerate}
\end{definition}

Intuitively, statistical order convergence relaxes the strict condition of classical order convergence by allowing the inequality to hold on a set of natural density one, rather than for all indices.
\begin{theorem}\label{thm:implication-soc}
Let $(R_n)$ be a sequence in $\mathcal{L}_b(\mathcal{L}, \mathcal{M})$. If $R_n \soc R$, then $R_n \socp R$.
\end{theorem}

\begin{proof}
Assume $R_n \soc R$. Then there exists a sequence $S_n \doc \theta$ and a subset $J \subseteq \mathbb{N}$ with $\delta(J)=1$ such that $|R_j-R| \leq S_j$ for all $j \in J$. Fix an arbitrary $u \in \mathcal{L}_+$. Since the inequality holds in the operator order, we have
$$
|R_j-R|(u) \leq S_j(u)
$$
for all $j \in J$. Recalling the modulus inequality $|S(u)| \le |S|(u)$ (see, e.g., \cite[Thm. 1.18]{AB1}), we obtain
$$
|R_j(u)-R(u)| = |(R_j - R)(u)| \le |R_j - R|(u) \leq S_j(u)
$$
for all $j \in J$. By Theorem \ref{implication}, the convergence $S_n \doc \theta$ implies $S_n \dopc \theta$. Thus, the sequence $(S_n)$ serves as the required control sequence for the definition of statistical order pointwise convergence. Therefore, we obtain $R_n \socp R$.
\end{proof}

The converse of Theorem \ref{thm:implication-soc} does not hold in general; that is, statistical order pointwise convergence does not imply statistical order convergence.
\begin{example}\label{ex:counter-socp-not-soc}
Consider the Riesz spaces $\mathcal{L} := c_0$, the space of null sequences, and $\mathcal{M} := \mathbb{R}$. Then the order dual $\mathcal{L}^\sim = \mathcal{L}_b(c_0, \mathbb{R})$ is isometrically Riesz isomorphic to $\ell_1$ (see, e.g., \cite[p. 67]{AB2}). Let $(R_n)$ be the sequence of coordinate functionals in $\mathcal{L}_b(c_0, \mathbb{R})$ defined by $R_n(x) = x_n$ for all $x=(x_k) \in c_0$.

First, we show that $R_n \socp \theta$. Fix any $u \in \mathcal{L}_+$. Then $R_n(u) = u_n$. Since $u \in c_0$, we have $u_n \to 0$ in $\mathbb{R}$. Moreover, we can explicitly define $S_n(u) := \sup_{m \ge n} |R_m(u)|$. This sequence satisfies $S_n(u) \downarrow 0$ and $|R_n(u)| \le S_n(u)$ for all $n$. Thus, $R_n(u) \xrightarrow{o} 0$, which implies $R_n \socp \theta$.

Now suppose, for contradiction, that $R_n \soc \theta$. Then there exist a sequence $(W_n)$ in $\ell_1$ with $W_n \doc \theta$ and a subset $J \subseteq \mathbb{N}$ with $\delta(J)=1$ such that $|R_j| \le W_j$ for every $j \in J$. Since $W_n \doc \theta$, there exists a subset $K \subseteq \mathbb{N}$ with $\delta(K)=1$ such that $W_k \downarrow \theta$ on $K$. Let $M := J \cap K$. Since the intersection of two sets of density one has density one, $\delta(M)=1$. Let $m_0 = \min M$. Since $(W_k)$ is decreasing on $M$, for any $n \in M$ with $n \ge m_0$, we have
$$
W_n \le W_{m_0}.
$$
Recall that $|R_n|$ corresponds to the standard basis vector $e_n$ in $\ell_1$. Thus, for $n \in M$, we have $e_n = |R_n| \le W_n \le W_{m_0}$. This implies that the $n$-th component of the vector $W_{m_0} \in \ell_1$ is at least $1$ for infinitely many indices $n$ because $M$ is infinite. Consequently, the series $\sum_{k=1}^\infty |(W_{m_0})_k|$ diverges, which contradicts the fact that $W_{m_0} \in \ell_1$. Thus, $(R_n)$ does not statistically order converge to zero.
\end{example}

\begin{lemma}\label{lem:linearity-soc}
Let $(R_n)$ and $(T_n)$ be sequences in $\mathcal{L}_b(\mathcal{L}, \mathcal{M})$ such that $R_n \soc R$ and $T_n \soc T$, and let $\alpha, \beta \in \mathbb{R}$. Then $\alpha R_n + \beta T_n \soc \alpha R + \beta T$.
\end{lemma}

\begin{theorem}\label{thm:lattice-soc}
Let $(R_n)$ and $(T_n)$ be sequences in $\mathcal{L}_b(\mathcal{L}, \mathcal{M})$ such that $R_n \soc R$ and $T_n \soc T$. Then the following hold:
\begin{enumerate}[(i)]
\item $|R_n| \soc |R|$,
\item $R_n \vee T_n \soc R \vee T$,
\item $R_n \wedge T_n \soc R \wedge T$.
\end{enumerate}
\end{theorem}

\begin{proof}
$(i)$ From the assumption $R_n \soc R$, there exist a sequence $S_n \doc \theta$ and a subset $J \subseteq \mathbb{N}$ with $\delta(J)=1$ such that $|R_j - R| \leq S_j$ for all $j \in J$. Using the basic lattice inequality $||x| - |y|| \leq |x - y|$, we obtain
$$
||R_j| - |R|| \leq |R_j - R| \leq S_j \quad \text{for every } j \in J.
$$
Since $S_n \doc \theta$, we conclude that $|R_n| \soc |R|$.

$(ii)$ We use the inequality $|(x \vee y) - (a \vee b)| \leq |x - a| + |y - b|$ (see, e.g., \cite[Thm. 12.4(ii)]{LZ}). Since $R_n \soc R$ and $T_n \soc T$, there exist sequences $S_n \doc \theta$, $Q_n \doc \theta$ and sets $J, K \subseteq \mathbb{N}$ with density one such that
$$
|R_j - R| \leq S_j \quad \text{and} \quad |T_k - T| \leq Q_k 
$$
for all $j \in J$ and $k \in K$. Let $M := J \cap K$. Then $\delta(M)=1$, and for any $m \in M$, we have
$$
|(R_m \vee T_m) - (R \vee T)| \leq |R_m - R| + |T_m - T| \leq S_m + Q_m.
$$
Define $H_n := S_n + Q_n$. By Lemma \ref{lem:additivity}, since $S_n \doc \theta$ and $Q_n \doc \theta$, we have $H_n \doc \theta$. Thus, $R_n \vee T_n \soc R \vee T$.

$(iii)$ The proof follows similar arguments to (ii) by using the inequality $|(x \wedge y) - (a \wedge b)| \leq |x - a| + |y - b|$.
\end{proof}

\begin{corollary}\label{cor:positive-negative-parts}
If $(R_n)$ is a sequence in $\mathcal{L}_b(\mathcal{L}, \mathcal{M})$ satisfying $R_n \soc R$, then we have $R_n^+ \soc R^+$ and $R_n^- \soc R^-$.
\end{corollary}

It is evident that any order convergent sequence in $\mathcal{L}_b(\mathcal{L}, \mathcal{M})$ is necessarily statistically order convergent. However, the converse implication fails in general, as demonstrated by the following example.
\begin{example}\label{ex:converse-fails}
Let $\mathcal{L} = \mathcal{M} = \mathbb{R}^2$ equipped with the coordinate-wise ordering. Let $I$ denote the identity operator on $\mathbb{R}^2$ and define the sequence of operators $(R_n)$ by:
$$
R_n :=
\begin{cases}
nI, & n = k^2 \text{ for some } k \in \mathbb{N}, \\
\frac{1}{n+1}I, & \text{otherwise}.
\end{cases}
$$
The sequence $(R_n)$ does not order converge to $\theta$ because the subsequence $(R_{k^2})$ is unbounded in the operator order. However, the set of indices $J:=\{n \in \mathbb{N} : n \neq k^2\}$ has natural density one.

Define a sequence $(S_n)$ in $\mathcal{L}_b(\mathcal{L}, \mathcal{M})$ by $S_n = \frac{1}{n}I$ for all $n \in \mathbb{N}$. Clearly, $S_n \downarrow \theta$ in $\mathcal{L}_b(\mathcal{L}, \mathcal{M})$, which implies $S_n \doc \theta$. Moreover, for any $j \in J$, we have $R_j = \frac{1}{j+1}I$. Since $I$ is a positive operator, we have 
$$
|R_j - \theta| = R_j = \frac{1}{j+1}I \leq \frac{1}{j}I = S_j.
$$
Therefore, the inequality $|R_j - \theta| \leq S_j$ holds for all $j \in J$, a set of density one. Thus, we conclude that $R_n \soc \theta$.
\end{example}

\begin{theorem}\label{thm:monotonicity-soc}
If $(R_n)$ is an increasing sequence in $\mathcal{L}_b(\mathcal{L},\mathcal{M})$ and $R_n \soc R$, then $R_n \xrightarrow{o} R$.
\end{theorem}

\begin{proof}
Assume that $R_n \soc R$. By the subsequence characterization of statistical order convergence, there exists a subset $J= \{j_1 < j_2 < \dots\} \subseteq \mathbb{N}$ with $\delta(J)=1$ such that the subsequence $(R_j)$ order converges to $R$. By hypothesis, the sequence $(R_n)$ is increasing, i.e., $R_n \uparrow$, we have $R_j \uparrow$ because any subsequence of an increasing sequence is also increasing. Recalling that the order limit of an increasing sequence coincides with its supremum, we obtain
$$
R= \sup_{j\in J} R_j.
$$
Now, fix an arbitrary $n \in \mathbb{N}$. Since the set $J$ has natural density one, it is infinite and unbounded. Thus, we can choose an index $j_0 \in J$ such that $j_0 \ge n$. Using the monotonicity of $(R_n)$, we have:
$$ 
R_n \leq R_{j_0} \leq \sup_{j\in J} R_j =R. 
$$
This inequality $R_n \leq R$ holds for all $n \in \mathbb{N}$, which implies that $R$ is an upper bound for the sequence $(R_n)$. On the other hand, since $\{R_j:j\in J\} \subseteq \{R_n : n \in \mathbb{N}\}$, the supremum of the subsequence cannot exceed the supremum of the entire sequence. Thus, we have
$$
R= \sup_{j \in J} R_j \leq \sup_{n \in \mathbb{N}} R_n \leq R.
$$
This implies $\sup_{n} R_n = R$. Consequently, since $(R_n)$ is increasing and its supremum is $R$, we conclude that $R_n \uparrow R$, which means $R_n\oc R$.
\end{proof}

The following theorem establishes the uniqueness of the statistical order limit, which is a fundamental requirement for any convergence notion.
\begin{theorem}\label{thm:soc-uniqueness}
Let $(R_n)$ be a sequence in $\mathcal{L}_b(\mathcal{L}, \mathcal{M})$. If $R_n \soc R_1$ and $R_n \soc R_2$, then $R_1=R_2$.
\end{theorem}

\begin{proof}
Suppose $R_n \soc R_1$ and $R_n \soc R_2$. By definition, there exist sequences $(S_n)$ and $(Q_n)$ in $\mathcal{L}_b(\mathcal{L}, \mathcal{M})$ with $S_n \doc \theta$, $Q_n \doc \theta$ and subsets $J, K \subseteq \mathbb{N}$ with $\delta(J) = \delta(J) = 1$ such that
$$
|R_j- R_1| \leq S_j \quad \text{and} \quad |R_k - R_2| \leq Q_k
$$
for all $j \in J$ and for all $k \in K$, respectively. Let $M:=J\cap K$. Since the intersection of two sets of density one has density one, we have $\delta(M)=1$. For any $m \in M$, using the triangle inequality, we have:
$$
|R_1 - R_2| \leq |R_1 - R_m| + |R_m - R_2| \leq S_m+ Q_m.
$$
Since $S_n \doc \theta$ and $Q_n \doc \theta$, it follows from Lemma \ref{lem:additivity} that $(S_n +Q_n) \doc \theta$. Hence, there exists a subset $N \subseteq \mathbb{N}$ with $\delta(N)=1$ such that $\inf_{n \in N} (S_n + Q_n) = \theta$. Let $D :=M\cap N$. Then $\delta(D)=1$, and for any $d \in D$, the inequality $|R_1 - R_2| \leq S_d + Q_d$ holds. Passing to the infimum over $d \in D$, we obtain
$$ 
|R_1 - R_2| \leq \inf_{d \in D} (S_d + Q_d) = \theta. 
$$
Thus, $|R_1 - R_2| = \theta$, which implies $R_1 = R_2$.
\end{proof}

The relationship between order convergence and statistical order convergence is characterized by the following theorem, adapted from the sequence case in \cite[Thm. 5]{SP}.
\begin{theorem}\label{thm:characterization}
Let $(R_n)$ be a sequence in $\mathcal{L}_b(\mathcal{L}, \mathcal{M})$. Then $R_n \soc R$ if and only if there exists a sequence $(T_n)$ in $\mathcal{L}_b(\mathcal{L}, \mathcal{M})$ such that $T_n \oc R$ and $\delta(\{n \in \mathbb{N} : R_n =T_n\}) = 1$.
\end{theorem}

\begin{proof}
Suppose such a sequence $(T_n)$ exists. Let $J:= \{n \in \mathbb{N}:R_n=T_n\}$. Since $T_n \oc R$, there exists a sequence $W_n \downarrow \theta$ such that $|T_n-R| \leq W_n$ for all $n \in \mathbb{N}$. Since order decreasing convergence implies statistical order decreasing convergence, we have $W_n \doc \theta$. Moreover, for every $j\in J$,
$$
|R_j-R| = |T_j - R| \leq W_j.
$$
Since $\delta(J)=1$, this implies $R_n \soc R$.

Assume $R_n \soc R$. Then there exist a sequence $S_n \doc \theta$ and a subset $J \subseteq \mathbb{N}$ with $\delta(J)=1$ such that $|R_j - R| \leq S_j$ for all $j \in J$. Since $S_n \doc \theta$, there exists a subset $M \subseteq \mathbb{N}$ with $\delta(M)=1$ such that $(S_m)$ decreases to $\theta$ on $M$. Let $D := J \cap M$. Then $\delta(D)=1$. Define a new sequence $(T_n)$ by
$$
T_n = \begin{cases} 
R_n, & n \in D, \\ 
R,   &  n \notin D.
\end{cases}
$$
Clearly $D \subseteq \{ n : R_n = T_n \}$, and hence $\delta(\{ n : R_n = T_n \}) = 1$. We now show that $T_n \oc R$. Define an auxiliary sequence $(Q_n)$ by
$$
Q_n = \begin{cases} 
S_n, & n \in D, \\ 
\theta, & n \notin D.
\end{cases}
$$
For every $n \in \mathbb{N}$, the inequality $|T_n - R| \leq Q_n$ holds. Indeed, if $n \in D$, then $|T_n - R| = |R_n - R| \leq S_n = Q_n$ and  if $n \notin D$, then $T_n = R$, and so we have  $|T_n - R| = \theta = Q_n$. Since $\mathcal{L}_b(\mathcal{L}, \mathcal{M})$ is Dedekind complete, the tail supremum
$$
Y_n := \sup_{k \ge n} Q_k
$$
is well-defined for all $n \in \mathbb{N}$. Since $(S_m)_{m \in M}$ decreases to zero and $D \subseteq M$, for each $n$, the set $\{ Q_k : k \ge n \}$ is bounded above by $S_{d_n}$, where $d_n := \min\{\, k \in D : k \ge n \,\}$. In fact, since $S_m$ decreases on $D$, we have $Y_n = S_{d_n}$. As $n \to \infty$, we must have $d_n \to \infty$. Since $S_{d_n} \downarrow \theta$, it follows that $Y_n \downarrow \theta$. Finally, from the inequalities $|T_n - R| \leq Q_n \leq Y_n$ and the fact that $Y_n \downarrow \theta$, we conclude that $T_n \oc R$.
\end{proof}

\begin{theorem}\label{thm:soc-squeeze}
Let $(R_n), (U_n), (T_n)$ be sequences in $\mathcal{L}_b(\mathcal{L}, \mathcal{M})$ such that $R_n \leq U_n \leq T_n$ for all $n \in \mathbb{N}$. If $R_n \soc R$ and $T_n \soc R$, then $U_n \soc R$.
\end{theorem}

\begin{proof}
From the hypothesis, the inequalities $\theta \leq U_n - R_n \leq T_n - R_n$ hold for all $n \in \mathbb{N}$. Since statistical order convergence is linear; see Lemma \ref{lem:linearity-soc}, the assumptions $R_n \soc R$ and $T_n \soc R$ imply that $T_n - R_n\soc R-R= \theta$. Consequently, the non-negative sequence $(U_n - R_n)$ is dominated by a sequence that statistically order converges to zero. By the definition of statistical order convergence, it follows that $U_n - R_n \soc \theta$. Finally, writing $U_n = (U_n - R_n) + R_n$ and applying linearity again, we obtain $U_n \soc \theta + R = R$.
\end{proof}

\begin{theorem}\label{thm:disjointness}
Let $(R_n)$ be a sequence in $\mathcal{L}_b(\mathcal{L},\mathcal{M})$ and $U\in \mathcal{L}_b(\mathcal{L},\mathcal{M})$. If $R_n \soc R$ and $R_n \perp U$ for all $n \in \mathbb{N}$, then $R \perp U$.
\end{theorem}

\begin{proof}
Assume that $R_n \soc R$ and $|R_n| \wedge |U| = \theta$ for all $n \in \mathbb{N}$. By Theorem \ref{thm:characterization}, there exists a subset $J\subseteq \mathbb{N}$ with $\delta(J)=1$ such that the subsequence $(R_j)$ order converges to $R$ in $\mathcal{L}_b(\mathcal{L},\mathcal{M})$. Since $\mathcal{M}$ is Dedekind complete, the lattice operations in $\mathcal{L}_b(\mathcal{L},\mathcal{M})$ are order continuous (see, e.g., \cite[Thm. 15.3]{LZ}). Hence, we obtain
$$
|R_j| \wedge |U| \oc |R| \wedge |U|
$$
as $j \to \infty$. By the hypothesis, $|R_j| \wedge |U| = \theta$ for all $j$, which implies that the order limit is zero. Therefore, we have $|R| \wedge |U| = \theta$. This shows that $R \perp U$.
\end{proof}

Recall that a linear subspace $\mathcal{B}$ of a Riesz space is called an {\em ideal} if $|x| \le |y|$ and $y \in \mathcal{B}$ imply $x \in \mathcal{B}$. An ideal $\mathcal{B}$ is called a {\em band} if it is order closed.
\begin{theorem}\label{thm:band-closed}
Let $\mathcal{B}$ be a band in $\mathcal{L}_b(\mathcal{L},\mathcal{M})$. If a sequence $(R_n)$ in $\mathcal{B}$ is statistically order convergent to an operator $R$, then $R \in \mathcal{B}$.
\end{theorem}

\begin{proof}
Suppose $(R_n)$ is a sequence in $\mathcal{B}$ and $R_n \soc R$. By Theorem \ref{thm:characterization}, there exists a subset $J \subseteq \mathbb{N}$ with $\delta(J)=1$ such that the subsequence $(R_j)$ order converges to $R$. Recall that a band in a Riesz space is an order closed ideal. Since $(R_j)$ is a sequence in $\mathcal{B}$ and it order converges to $R$, the order closedness of $\mathcal{B}$ forces the limit $R$ to belong to $\mathcal{B}$. Hence we have $R \in \mathcal{B}$.
\end{proof}

\begin{corollary}\label{cor:order-continuous}
Le $(R_n)$ be a sequence in the order continuous dual $\mathcal{L}_n(\mathcal{L},\mathcal{M})$. If $(R_n)$ statistically order converges to an operator $R \in \mathcal{L}_b(\mathcal{L},\mathcal{M})$, then $R$ is also order continuous.
\end{corollary}

\begin{proof}
Suppose that $R_n \soc R$ holds for a sequence $(R_n)$ in $\mathcal{L}_n(\mathcal{L},\mathcal{M})$. Since $\mathcal{M}$ is Dedekind complete, it is a known fact that $\mathcal{L}_n(\mathcal{L},\mathcal{M})$ is a band in the Riesz space $\mathcal{L}_b(\mathcal{L},\mathcal{M})$ (see, e.g., \cite[Thm. 84.2]{Za}). Therefore, it follows immediately from Theorem \ref{thm:band-closed} that $R\in \mathcal{L}_n(\mathcal{L},\mathcal{M})$. Hence, $R$ is order continuous.
\end{proof}

\begin{theorem}\label{thm:composition-soc}
Let $(R_n)$ be a sequence in $\mathcal{L}_b(\mathcal{L}, \mathcal{M})$ with $R_n \soc R$, and let $T: \mathcal{M} \to \mathcal{N}$ be a positive $\sigma$-order continuous operator. Then $T \circ R_n \soc T \circ R$.
\end{theorem}

\begin{proof}
Since $R_n \soc R$, there exist a sequence $S_n\doc\theta$ and a subset $J \subseteq \mathbb{N}$ with $\delta(J)=1$ such that $|R_j - R| \leq S_j$ for all $j \in J$. Since $T$ is a positive operator, it satisfies the inequality $|T \circ U| \leq T \circ |U|$ for any operator $U$. Using the linearity of $T$, we obtain:
$$
|T \circ R_j - T \circ R| = |T \circ (R_j - R)| \leq T \circ |R_j - R| \leq T \circ S_j
$$
for all $j \in J$. Define the sequence $Q_n := T \circ S_n$. Since $(S_n)$ decreases to $\theta$ and $T$ is $\sigma$-order continuous, it follows that $Q_n \downarrow \theta$; see Theorem \ref{thm:composition}$(ii)$. Moreover, the inequality $|T \circ R_j - T \circ R| \leq P_j$ holds on the set $J$. Thus, we conclude that $T \circ R_n \soc T \circ R$.
\end{proof}

\begin{theorem}\label{thm:band-projection}
Let $P:\mathcal{M} \to \mathcal{M}$ be a band projection. If $(R_n)$ is a sequence in $\mathcal{L}_b(\mathcal{L},\mathcal{M})$ such that $R_n \soc R$, then $P \circ R_n \soc P \circ R$.
\end{theorem}

\begin{proof}
Assume $R_n \soc R$. Then there exist a sequence $(S_n)$ in $\mathcal{L}_b(\mathcal{L},\mathcal{M})$ with $S_n \doc \theta$ and a subset $J \subseteq \mathbb{N}$ with $\delta(J)=1$ such that $|R_j - R| \leq S_j$ for all $j \in J$. Recall that a band projection is positive and $\sigma$-order continuous (see, e.g., \cite[Thm. 1.44]{AB2}). By applying Theorem \ref{thm:composition}$(ii)$ to the operator $P$ and the sequence $(S_n)$, we obtain $P \circ S_n \doc \theta$.

Since every band projection is a lattice homomorphism (see, e.g., \cite[p. 94]{AB2}), the equality $|P \circ U| = P \circ |U|$ holds for any operator $U$. Using this property and the positivity of $P$, we obtain:
$$
|P \circ R_j - P \circ R| = |P \circ (R_j - R)| = P \circ |R_j - R| \leq P \circ S_j.
$$
for all indices $j \in J$. This implies that $P \circ R_n \soc P \circ R$.
\end{proof}

The paper is concluded with the introduction of statistically order boundedness for operator sequences and the establishment of essential related results.
\begin{definition}\label{def:stat_ob}
Let $(R_n)$ be a sequence in $\mathcal{L}_b(\mathcal{L},\mathcal{M})$. We say that $(R_n)$ is \textit{statistically order bounded} if there exists an operator $S\in\mathcal{L}_b(\mathcal{L},\mathcal{M})$ and a subset $J \subseteq \mathbb{N}$ with natural density $\delta(J)=1$ such that $|R_j|\leq S$ for every $j \in J$.
\end{definition}

It is clear that every order bounded sequence is statistically order bounded, but the converse is not true in general.
\begin{theorem}\label{thm:st_conv_implies_st_ob}
If a sequence $(R_n)$ in $\mathcal{L}_b(\mathcal{L},\mathcal{M})$ is statistically order convergent, then it is statistically order bounded.
\end{theorem}

\begin{proof}
Assume that $R_n \soc R$. By definition, there exist a sequence $(S_n)$ in $\mathcal{L}_b(\mathcal{L},\mathcal{M})$ with $S_n \doc \theta$ and a subset $J \subseteq \mathbb{N}$ with $\delta(J)=1$ such that
$$
|R_j - R| \leq S_j 
$$
for all $j \in J$. Since $S_n \doc \theta$, there exists a subset $K \subseteq \mathbb{N}$ with $\delta(K)=1$ such that $(S_k)$ decreases to $\theta$ on $K$. Let $M:=J\cap K$. Then $\delta(M)=1$. Since $(S_k)$ is decreasing on $K$, the element $S_{k_0}$ for $k_0 = \min K$ serves as an upper bound for the tail. Thus, for every $m \in M$, we have
$$
||R_m| - |R||\leq |R_m - R| \leq S_m \leq S_{k_0}.
$$
This implies $|R_m| \leq |R| + S_{k_0}$ for all $m \in M$. Taking $S:= |T| + S_{k_0}$, we see that $(R_n)$ is bounded by $S$ on the set $M$ of density one. Hence, $(R_n)$ is statistically order bounded.
\end{proof}

\begin{theorem}\label{thm:decomposition}
Let $(R_n)$ be a statistically order bounded sequence in $\mathcal{L}_b(\mathcal{L},\mathcal{M})$. Then there exist an order bounded sequence $(T_n)$ and a sequence $(U_n)$ with $\delta(\{n \in \mathbb{N}: U_n \neq \theta\}) = 0$ such that $R_n = T_n + U_n$ for every $n \in \mathbb{N}$.
\end{theorem}

\begin{proof}
Since $(R_n)$ is statistically order bounded, there exists an operator $S \in \mathcal{L}_b(\mathcal{L},\mathcal{M})_+$ and a subset $J \subseteq \mathbb{N}$ with $\delta(J)=1$ such that $|T_j| \leq S$ for all $j \in J$. Let $K := \mathbb{N} \setminus J$. Then $\delta(K)=0$. Define two sequences $(T_n)$ and $(U_n)$ by:
$$
T_n = \begin{cases}
R_n, & \text{if } n \in J,\\
\theta,   & \text{if } n \in K,
\end{cases}
\qquad \text{and} \qquad
U_n = \begin{cases}
\theta,       & \text{if } n \in J,\\
R_n, & \text{if } n \in K.
\end{cases}
$$
It is clear that $R_n = T_n + U_n$ for every $n \in \mathbb{N}$. The sequence $(T_n)$ is order bounded because $|T_n| \leq S$ for all $n \in \mathbb{N}$. Moreover, $S_n \neq \theta$ can only happen when $n \in K$. Since $\delta(K)=0$, we have $\delta(\{n : S_n \neq \theta\}) = 0$.
\end{proof}

\begin{corollary}\label{cor:decomposition_st_conv}
Every statistically order convergent sequence in $\mathcal{L}_b(\mathcal{L},\mathcal{M})$ can be written as the sum of an order convergent sequence and a sequence that vanishes on a set of natural density one.
\end{corollary}

\section{Conclusions}
In this paper, we have developed the theory of statistical order convergence for sequences of order bounded operators on Riesz spaces. The results demonstrate that this notion extends classical order convergence in a non-trivial way while preserving many structural properties of operator lattices. Possible directions for future research include statistical unbounded order convergence for operators and applications to Banach lattice theory.



\begin{thebibliography}{99}
\bibitem{AB1} Aliprantis, C. D.; Burkinshaw, O. {\it Locally Solid Riesz Spaces with Applications to Economics}. Mathematical Surveys and Monographs, Vol. 105, American Mathematical Society, Providence, RI, 2003.

\bibitem{AB2} Aliprantis, C. D.; Burkinshaw, O. {\it Positive Operators}. Springer, Dordrecht, 2006.

\bibitem{Aydn1} Ayd{\i}n, A. Statistically unbounded p-convergence in lattice-normed Riesz spaces. {\em Miskolc Math. Notes} {\bf 24} (2023), no. 3, 1185--1196.

\bibitem{Aydn2} Ayd{\i}n, A. The statistical multiplicative order convergence in vector lattice algebras. {\em Facta Univ. Ser. Math. Inform.} {\bf 36} (2021), no. 2, 409--417.

\bibitem{AE} Ayd{\i}n, A.; Et, M. Statistically multiplicative convergence on locally solid Riesz algebras. {\em Turkish J. Math.} {\bf 45} (2021), no. 4, 1506--1516.

\bibitem{AEG} Ayd{\i}n, A.; Emelyanov, E.; Gorokhova, S. G. Full lattice convergence on Riesz spaces. {\em Indag. Math. (N.S.)} {\bf 32} (2021), no. 3, 658--690.

\bibitem{BDK} Balcerzak, M.; Dems, K.; Komisarski, A. Statistical convergence and ideal convergence for sequences of functions. {\em J. Math. Anal. Appl.} {\bf 328} (2007), no. 1, 715--729.

\bibitem{BW} Bayram, E.; Wickstead, A. W. Banach lattices of $L$-weakly and $M$-weakly compact operators. {\em Arch. Math. (Basel)} {\bf 108} (2017), 293--299.

\bibitem{Et} Braha, N. L.; Srivastava, H. M.; Et, M. Some weighted statistical convergence and associated Korovkin and Voronovskaya type theorems. {\em J. Appl. Math. Comput.} {\bf 65} (2021), no. 1-2, 429--450.

\bibitem{Chit} Chittenden, E. W. On the limit functions of sequences of continuous functions converging relatively uniformly. {\em Trans. Amer. Math. Soc.} {\bf 20} (1919), 179--184.

\bibitem{Con} Connor, J. Two valued measures and summability. {\em Analysis} {\bf 10} (1990), no. 4, 373--385.

\bibitem{DO} Duman, O.; Orhan, C. $\mu$-statistically convergent function sequences. {\em Czechoslovak Math. J.} {\bf 54} (2004), no. 2, 413--422.

\bibitem{EDS} Erku\c{s}, E.; Duman, O.; Srivastava, H. M. Statistical approximation of certain positive linear operators constructed by means of the Chan--Chyan--Srivastava polynomials. {\em Appl. Math. Comput.} {\bf 182} (2006), no. 1, 213--222.

\bibitem{Et1} Et, M.; Baliarsingh, P.; Kandemir, H. \c{S}.; K\"u\c{c}\"ukaslan, M. On $\mu$-deferred statistical convergence and strongly deferred summable functions. {\em Rev. R. Acad. Cienc. Exactas F\'is. Nat. Ser. A Mat. RACSAM} {\bf 115} (2021), no. 1, 34.

\bibitem{Fast} Fast, H. Sur la convergence statistique. {\em Colloq. Math.} {\bf 2} (1951), 241--244.

\bibitem{Fridy} Fridy, J. A. On statistical convergence. {\em Analysis} {\bf 5} (1985), no. 4, 301--313.

\bibitem{Gor} Gorokhova, S. G. Intrinsic characterization of the space $c_0(A)$ in the class of Banach lattices. {\em Mat. Zametki} {\bf 60} (1996), no. 3, 330--333.

\bibitem{KDD} Karaku\c{s}, S.; Demirci, K.; Duman, O. Equi-statistical convergence of positive linear operators. {\em J. Math. Anal. Appl.} {\bf 339} (2008), no. 2, 1065--1072.

\bibitem{LZ} Luxemburg, W. A. J.; Zaanen, A. C. {\it Riesz Spaces I}. North-Holland Publishing Company, Amsterdam, 1971.

\bibitem{Moor} Moore, E. H. {\it An Introduction to a Form of General Analysis}. Yale University Press, New Haven, 1910.

\bibitem{Mor} M\'oricz, F. Statistical limits of measurable functions. {\em Analysis} {\bf 24} (2004), no. 1, 1--18.

\bibitem{Riez} Riesz, F. Sur la d\'ecomposition des op\'erations fonctionelles lin\'eaires. In {\it Atti del Congresso Internazionale dei Matematici (Bologna, 1928)}, Vol. 3, 143--148, Nicola Zanichelli, Bologna, 1929.

\bibitem{SP} \c{S}en\c{c}imen, C.; Pehlivan, S. Statistical order convergence in Riesz spaces. {\em Math. Slovaca} {\bf 62} (2012), no. 2, 557--570.

\bibitem{St} Steinhaus, H. Sur la convergence ordinaire et la convergence asymptotique. {\em Colloq. Math.} {\bf 2} (1951), 73--74.

\bibitem{Trip} Tripathy, B. C. On statistically convergent sequences. {\em Bull. Calcutta Math. Soc.} {\bf 90} (1988), 259--262.

\bibitem{Vu} Vulikh, B. Z. {\it Introduction to the Theory of Partially Ordered Spaces}. Wolters-Noordhoff Ltd, Groningen, 1967.

\bibitem{Za} Zaanen, A. C. {\it Riesz Spaces II}. North-Holland Publishing, Amsterdam, 1983.
\end{thebibliography}
\end{document}